\documentclass[10pt,twocolumn,letterpaper]{article}

\usepackage{times}
\usepackage{epsfig}
\usepackage{graphicx}
\usepackage{amsmath}
\usepackage{amssymb}
\usepackage{amsthm}
\usepackage{booktabs}
\usepackage{subfig}
\usepackage{tikz} % NOTE eso-pic.sty provided by CVPR _MUST_ be removed for this to work
\usepackage{pgfplots}
\usepackage{multirow}
\usepackage{bm}
\usepackage{cvpr}
\usepackage{dsfont}
\usepackage{array}

\usepackage[pagebackref=true,breaklinks=true,letterpaper=true,colorlinks,bookmarks=false]{hyperref}

\cvprfinalcopy % *** Uncomment this line for the final submission

\def\cvprPaperID{550} % *** Enter the CVPR Paper ID here

\DeclareMathOperator*{\argmin}{arg\,min}

\ifcvprfinal\fi
\begin{document}

\def\vv{{\bf v}}
\def\mm{{\bf m}}
\def\lambdabf{{\bm \lambda}}
\def\x{{\bf x}}
\def\y{{\bf x}'}
\def\yy{{\bf y}}
\def\z{{\bf z}}
\def\bb{{\bf b}}
\def\bfepsilon{{\bm \epsilon}}
\def\Y{{X'}}
\def\rank{{\text{ rank}}}
\def\tr{{\text{ tr}}}
\def\supp{{\text{ card}}}
\def\diag{{\text{ diag}}}
\def\prox{{\text{ prox}}}
\def\sgn{{\text{ sgn}}}
\def\mat{{\mathcal{M}}}
\def\A{{\mathcal{A}}}
\def\bfsigma{{\bm \sigma}}
\newcommand{\skal}[1]{\langle #1 \rangle}
\newcommand{\trunc}[1]{\left[ #1 \right]_+}
\newcommand{\m}[1]{\begin{bmatrix} #1 \end{bmatrix}}
\newcommand{\reals}{\ensuremath{\mathbb{R}}}
\newcommand{\pspace}{\ensuremath{\mathbb{P}}}
\renewcommand{\P}{\ensuremath{\mathcal{P}}}
\newcommand{\dX}{\ensuremath{\Delta X}}
\newcommand{\reg}{\ensuremath{\mathcal{R}}}
\renewcommand{\vec}[1]{\text{vec}\left( #1 \right)}
\def\L{\mathcal{L}}
\def\e{{\mathds{1}}}
\newcommand{\bmat}[1]{\begin{bmatrix}#1 \end{bmatrix}}

\newtheorem{theorem}{Theorem}[section]
\newtheorem{lemma}[theorem]{Lemma}
\newtheorem{proposition}[theorem]{Proposition}
\newtheorem{corollary}[theorem]{Corollary}
\newtheorem{remark}[theorem]{Remark}

%==========================
%%%%%%%%% TITLE
\title{Non-Convex Rank/Sparsity Regularization and Local Minima}

\author{Carl Olsson, Marcus Carlsson, Fredrik Andersson, Viktor Larsson\\
Centre for Mathematical Science\\
Lund University\\
{\tt\small \{calle,mc,fa,viktorl\}@maths.lth.se}
}

\maketitle
%\thispagestyle{empty}

%===========================================================
\begin{abstract}
This paper considers the problem of recovering either a low rank matrix or a sparse vector from observations of linear combinations of the vector or matrix elements.
Recent methods replace the non-convex regularization with $\ell_1$ or nuclear norm relaxations.
It is well known that this approach can be guaranteed to recover a near optimal solutions
if a so called restricted isometry property (RIP) holds.
On the other hand it is also known to perform soft thresholding which results in a shrinking bias which can degrade the solution.

In this paper we study an alternative non-convex regularization term.
This formulation does not penalize elements that are larger than a certain threshold making it much less prone to small solutions.
Our main theoretical results show that if a RIP holds then the stationary points are often well separated, in the sense that their differences must be of high cardinality/rank.
Thus, with a suitable initial solution the approach is unlikely to fall into a bad local minima.
Our numerical tests show that the approach is likely to converge to a better solution than standard $\ell_1$/nuclear-norm relaxation even when starting from trivial initializations.
In many cases our results can also be used to verify global optimality of our method.
\end{abstract}

%===========================================================

\section{Introduction}
Sparsity penalties are important priors for regularizing linear systems.
Typically one tries to solve a formulation that minimizes a trade-off between sparsity and residual error such as
\begin{equation}
\mu \supp(\x) +  \|A \x- \bb \|^2,
\label{eq:sparsityobjektive}
\end{equation}
where $\supp(\x)$ is the number of non-zero elements in $\x$, and the matrix $A$ is of size $m \times n$.
Direct minimization of \eqref{eq:sparsityobjektive} is generally considered difficult because of the properties of the $\supp$ function, which is non-convex and discontinuous. The method that has by now become the standard approach is to replace $\supp(\x)$ with the convex $\ell_1$ norm $\|\x\|_1$ \cite{tropp-2015,tropp-2006,candes-etal-cpam-2006,candes-tao-2006,donoho-elad-2002}. This choice can be justified with the $\ell_1$ norm being the convex envelope of the $\supp$ function on the set $\{\x;\|\x\|_\infty \leq 1\}$.
Furthermore, strong performance guarantees can be derived \cite{candes-etal-cpam-2006,candes-tao-2006} if $A$ obeys a RIP
\begin{equation}
(1-\delta_c)\|\x\|^2 \leq \|A \x\|^2 \leq (1+\delta_c)\|\x\|^2,
\label{eq:vectorRIP}
\end{equation}
for all vectors $\x$ with $\supp(\x) \leq c$, where $c$ is a bound on the number of non-zero terms in the sought solution.
The $\ell_1$ approach does however suffer from a shrinking bias.
In contrast to the $\supp$ function the $\ell_1$ norm penalizes both small elements of $\x$, assumed to stem from measurement noise, and large elements, assumed to make up the true signal, equally.
In some sense the suppression of noise also requires an equal suppression of signal.
Therefore non-convex alternatives able to penalize small components proportionally harder have been considered \cite{daubechies-etal-cpam-2010,candes-etal-2008}.
On the downside convergence to the global optimum is not guaranteed.

This paper considers the non-convex relaxation
\begin{equation}
r_\mu(\x) + \|A\x-\bb\|^2,
\label{eq:vectorrelaxation}
\end{equation}
where $r_\mu(\x) = \sum_i \left( \mu-\max(\sqrt{\mu}-|x_i|,0)^2\right)$. Figure~\ref{fig:threeregterms} shows one dimensional illustrations of the $\supp$-function, $\ell_1$-norm and $r_\mu$ term.
It can be shown \cite{jojic2011convex,soubies-etal-siims-2015,larsson-olsson-ijcv-2016} that the convex envelope of
\begin{equation}
\mu \supp(\x) + \|\x - \z \|^2,
\label{eq:suppgoal}
\end{equation}
where $\z$ is some given vector, is
\begin{equation}
r_\mu(\x)+\|\x-\z\|^2.
\label{eq:supprelax}
\end{equation}
Note that similarly to the $\supp$ function $r_\mu$ does not penalize elements that are larger than $\sqrt{\mu}$. In fact it is easy to show that the minimizer $\x^*$ of both \eqref{eq:suppgoal} and \eqref{eq:supprelax} is given by thresholding of $\z$, that is, $x^*_i = z_i$ if $|z_i| > \sqrt{\mu}$ and $x^*_i = 0$ if $|z_i| < \sqrt{\mu}$. If there is an $i$ such that $|z_i|=\sqrt{\mu}$ then the minimizer is not unique. In \eqref{eq:suppgoal} $x^*_i$ can take either the value $0$ or $\sqrt{\mu}$
and in \eqref{eq:supprelax} any convex combination of these.
The regularization term $r_\mu$ is by itself not convex, see Figure~\ref{fig:threeregterms}. However when combined with a quadratic term $\|\x-\z\|^2$, non-convexities are canceled and the result is a convex objective function.
\begin{figure}[htb]
\begin{center}
\includegraphics[width=60mm]{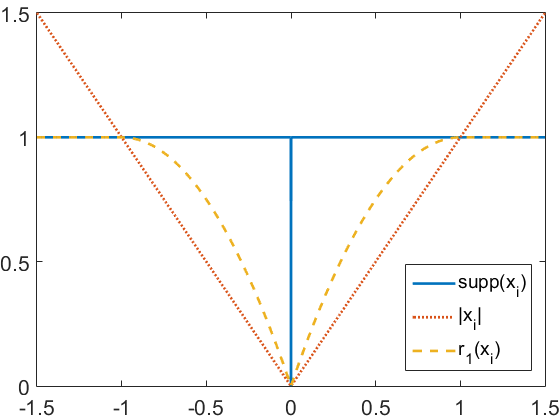}
\end{center}
\caption{One dimensional illustrations of the three regularization terms (when $\mu=1$).}
\label{fig:threeregterms}
\end{figure}

Assuming that $A$ fulfills \eqref{eq:vectorRIP} it is natural to wonder about convexity properties of \eqref{eq:vectorrelaxation}. Intuitively $\|A\x\|^2$ seems to behave like $\|\x\|^2$ which combined with $r_\mu(\x)$ only has one local minimum. In this paper we make this reasoning formal and study the stationary points of \eqref{eq:vectorrelaxation}.
We show that if $\x_s$ is a stationary point of \eqref{eq:vectorrelaxation} and the elements of the vector $\z = (I-A^T A)\x_s+A^T\bb$ fulfill $|z_i| \notin [\sqrt{\mu}(1-\delta_c), \frac{\sqrt{\mu}}{1-\delta_c}]$ then for any other stationary point $\y_s$ we have $\supp(\x_s-\y_s)>c$. A simple consequence is that if we for example find such a local minimizer with $\supp(\x_s) < c/2$ then this is the sparsest possible one.

The meaning of the vector $\z$ can in some sense be understood by seeing that
the stationary point $\x_s$ fulfills $\x_s \in \argmin_\x r_\mu(\x) + \|\x - \z\|^2$, which we prove in Section \ref{sec:sparsity}. 
Hence $\x_s$ can be obtained through thresholding of the vector $\z$.
Our results then essentially state that if the elements of $\z$ are not to close to the thresholds $\pm \sqrt{\mu}$ then $\supp(\x_s-\y_s)>c$ holds for all other stationary points $\y_s$.

In two very recent papers \cite{soubies-etal-siims-2015,carlsson2016convexification} the relationship between (both local and global) minimizers of \eqref{eq:vectorrelaxation} and \eqref{eq:sparsityobjektive} is studied. Among other things \cite{carlsson2016convexification} shows that if $\|A\| \leq 1$ then any local minimizer of \eqref{eq:vectorrelaxation} is also a local minimizer of \eqref{eq:sparsityobjektive}, and that their global minimizers coincide.
Hence results about the stationary points of \eqref{eq:vectorrelaxation} are also relevant to the original discontinuous objective \eqref{eq:sparsityobjektive}.

The theory of rank minimization largely parallels that of sparsity with the elements $x_i$ of the vector $\x$ replaced by the singular values $\sigma_i(X)$ of the matrix $X$.
Typically we want to solve a problem of the type
\begin{equation}
\mu \rank (X) + \|\A X - \bb\|^2,
\label{eq:rankobjective}
\end{equation}
where $\A : \mathbb{R}^{m\times n} \mapsto \mathbb{R}^p$ is some linear operator on the set of $m\times n$ matrices.
In this context the standard approach is to replace the rank function with the convex nuclear norm $\|X\|_* = \sum_i \sigma_i(X)$ \cite{recht-etal-siam-2010,candes-etal-acm-2011}.
It was first observed that this is the convex envelope of the rank function over the set $\{X; \sigma_1(X) \leq 1\}$ in \cite{fazel-etal-acc-2015}.
In \cite{recht-etal-siam-2010} the notion of RIP was generalized to the matrix setting requiring that
$\A$ is a linear operator $\mathbb{R}^{m\times n} \rightarrow \mathbb{R}^k$ fulfilling
\begin{equation}
(1-\delta_r)\|X\|_F^2 \leq \|\A X\|^2 \leq (1+\delta_r)\|X\|_F^2,
\label{eq:matrisRIP}
\end{equation}
for all $X$ with $\rank(X) \leq r$.
Since then a number of generalizations that give performance guarantees for the nuclear norm relaxation have appeared \cite{oymak2011simplified,candes-etal-acm-2011,candes2009exact}.
Non-convex alternatives have also been shown to improve performance \cite{oymak-etal-2015,mohan2010iterative}.

Analogous to the vector setting it was recently shown \cite{larsson-olsson-ijcv-2016} that the convex envelope of
\begin{equation}
\mu \rank(X) + \| X - M \|_F^2,
\end{equation}
is given by
\begin{equation}
r_\mu (\bfsigma(X)) + \|X - M\|_F^2,
\end{equation}
where $\bfsigma(X)$ is the vector of singular values of $X$.
In \cite{andersson-carlsson-2016} an efficient fixed-point algorithm is developed for
objective functions of the type $r_\mu (\bfsigma(X)) + q\|X - M\|_F^2$ with linear constraints.
The approach is illustrated to work well even when $q < 1$ which gives a non-convex objective.

In this paper we consider
\begin{equation}
r_\mu (\bfsigma(X)) + \|\A X - \bb\|^2,
\label{eq:rankrelax}
\end{equation}
where $\A$ obeys \eqref{eq:matrisRIP}.
Our main result states that if $X_s$ is a stationary point of \eqref{eq:rankrelax}
and $Z = (I-\A^* \A)X_s+\A^*\bb$ has no singular values in the interval
$[\sqrt{\mu}(1-\delta_r), \frac{\sqrt{\mu}}{1-\delta_r}]$
then for any other stationary point we have $\rank(X_s-X_s') > r$.

\section{Notation and Preliminaries}
In this section we introduce some preliminary material and notation. 
In general we will use boldface to denote a vector $\x$ and its $i$th element $x_i$.
By $\|\x\|$ we denote the standard euclidean norm $\|\x\| = \sqrt{\x^T \x}$.
We use $\sigma_i(X)$ to denote the $i$th singular value of a matrix $X$. 
The vector of all singular values is denoted $\bfsigma(X)$. 
A diagonal matrix with diagonal elmenents $\x$ will be denoted $D_\x$.
The scalar product is defined as $\skal{X,Y} = \tr(X^T Y)$, where $\tr$ is the trace function, and the Frobenius norm $\|X\|_F = \sqrt{\skal{X,X}}=\sqrt{\sum_{i=1}^n\sigma_i^2(X)}$.
The adjoint of a the linear matrix operator $\A$ is denoted $\A^*$.
For functions taking values in $\mathbb{R}$ such as $r_\mu$ we will frequently use the convention that $r_\mu(\x) = \sum_i r_\mu(x_i)$.

The function $g(x) = r_\mu(x)+x^2$ will be useful when considering stationary points, since it is convex with a well defined sub-differential. We can write $g$ as
\begin{equation}
g(x) = 
\begin{cases}
\mu + x^2 & |x| \geq \sqrt{\mu} \\
2\sqrt{\mu}|x| & 0 \leq |x| \leq \sqrt{\mu}
\end{cases}.
\label{eq:gdef}
\end{equation}
Its sub-differential is given by
\begin{equation}
\partial g(x) = 
\begin{cases}
\{2 x\} & |x| \geq \sqrt{\mu} \\
\{2\sqrt{\mu}\text{sign}(x)\} & 0 < |x| \leq \sqrt{\mu} \\
[-2\sqrt{\mu},2\sqrt{\mu}] & x = 0
\end{cases}.
\label{eq:vectorsubgrad}
\end{equation}
Note that the sub-differential consists of a single point for each $x \neq 0$.
By $\partial g(\x)$ we mean the set of vectors $\{\z; z_i \in \partial g(x_i), \ \forall i\}$.
Figure~\ref{fig:gfunk} illustrates $g$ and its sub-differential.
\begin{figure}[htb]
\begin{center}
\includegraphics[width=40mm]{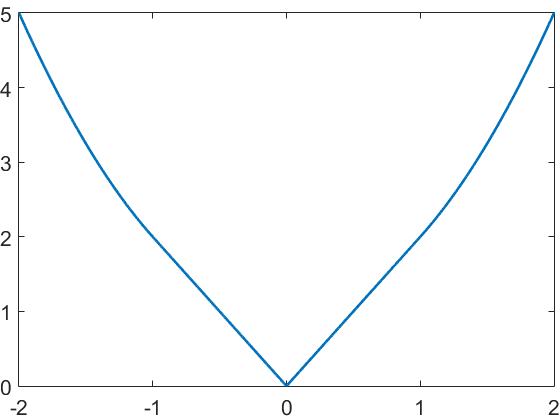} 
\includegraphics[width=40mm]{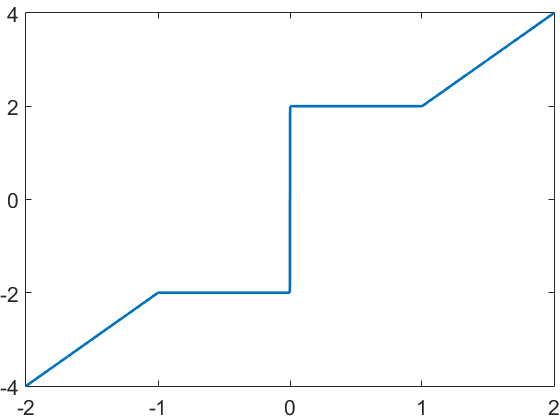} 
\end{center}
\caption{The function $g(x)$ (left) and its sub-differential $\partial g(x)$ (right). Note that the sub-differential contains a unique element everywhere except at $x=0$.}
\label{fig:gfunk}
\end{figure}
For the matrix case we similarly define $G(X) = r_\mu(\bfsigma(X)) + \|X\|_F^2$.
It can be shown \cite{lewis-95} that a matrix $Z$ is in the sub-differential of $G$ at $X$  if and only if
\begin{equation}
Z = U D_\z V^T, \text{ where } \z \in \partial g(\bfsigma(X)),
\end{equation}
where $U D_{\bfsigma(X)} V^T$ is the SVD of $X$.
Roughly speaking the sub-differential at $X$ is obtained by taking the sub-differential at each singular value.

In Section~\ref{sec:rank} we utilize the notion of doubly sub-stochastic (DSS) matrices \cite{andersson-etal-pams-2016}.
A matrix $M$ is DSS if its rows and columns fulfill $\sum_i |m_{ij}| \leq 1$ and $\sum_j |m_{ij}| \leq 1$.
The DSS matrices are closely related to permutations. 
Let $\pi$ denote a permutation and $M_{\pi,\vv}$ the matrix with elements $m_{i,\pi(i)}=v_i$ and zeros otherwise.
It is shown in \cite{andersson-etal-pams-2016} (Lemma 3.1) that an $m \times m$ matrix is DSS if and only if it lies in the convex hull of the set
$\{M_{\pi,\vv}; \pi \text{ is a permutation}, \ |v_i|=1 \ \forall i \}$. The result is actually proven for matrices with complex entries, but the proof is identical for real matrices.

\section{Sparsity Regularization}\label{sec:sparsity}
In this section we consider stationary points of the proposed sparsity formulation.
Let
\begin{equation}
f(\x) = r_\mu(\x) + \|A\x-\bb\|^2.
\end{equation}
The function $f$ can equivalently be written
\begin{equation}
f(\x) = g(\x) + \x^T(A^T A - I)\x-2\x^TA^T\bb + \bb^T \bb.
\end{equation}
Taking derivatives we see that the stationary points solve
\begin{equation}
2(I -A^T A)\x_s+2A^T\bb \in \partial g(\x_s).
\label{eq:vectorstationary}
\end{equation}
The following lemma clarifies the connection between a stationary point $\x_s$ and the vector
$\z =  (I -A^T A)\x_s+A^T\bb$.

\begin{lemma}\label{lemma:lowrankstatpt}
The point $\x_s$ is stationary in $f$ if and only if $2\z \in \partial g(\x_s)$ and if and only if
\begin{equation}\label{ancud}
\x_s \in \argmin_\x r_\mu(\x) + \|\x -  \z\|^2.
\end{equation}
\end{lemma}
\begin{proof}
By \eqref{eq:vectorstationary} we know that $\x_s$ is stationary in $f$ if and only if
$2\z \in \partial g(\x_s)$. Similarly, inserting $A = I$ and $\bb = \z$ in \eqref{eq:vectorstationary} shows that $\x_s$ is stationary in $r_\mu (\x) + \left\|\x- \z\right\|^2$
if and only if
%\begin{equation}
$2\z \in \partial g(\x_s)$.
%\end{equation}
Since $r_\mu (\x) + \left\|\x- \z\right\|^2$ is convex in $\x$, this is equivalent to solving \eqref{ancud}.
\end{proof}

The above result shows that stationary points of $f$ are sparse approximations of $\z$ in the sense that small elements are suppressed. The elements of $\x_s$ are either zero or have magnitude larger than $\sqrt{\mu}$ assuming that the vector $\z$ has no elements that are precisely $\pm \sqrt{\mu}$. The term $\|\x - \z\|^2$ can also be seen as a local approximation of $\|A \x - \bb\|^2 =$
\begin{equation}
\|\x\|^2 - \x^T(I-A^T A)\x -2\bb^T A \x +\bb^T \bb.
\label{eq:expansion}
\end{equation}
Replacing $\x^T(I-A^T A)\x$ with its first order Taylor expansion $2\x^T(I-A^T A)\x_s$
(ignoring the constant term) reduces \eqref{eq:expansion} to
$\|\x-\z\|^2+C$, where $C$ is a constant.

\subsection{Stationary points under the RIP constraint}
We now assume that $A$ is a matrix fulfilling the RIP \eqref{eq:vectorRIP}.
We can equivalently write
\begin{equation}
f(\x)
 =  g(\x) - \delta_c\|\x\|^2
+ h(\x) +\|\bb\|^2,
\label{eq:vectorobj}
\end{equation}
where
\begin{equation}
h(\x) = \delta_c \|\x\|^2 + \left(\|A\x\|^2-\|\x\|^2\right)-2\x^T A^T \bb.
\end{equation}
The term $\|\bb\|^2$ is constant with respect to $\x$ and we can therefore drop it without affecting the optimizers.
The point $\x$ is a stationary point of $f$ if $2\delta_c \x -\nabla h(\x) \in \partial g(\x)$, that is there is a vector $2\z \in \partial g(\x)$ such that $2\delta_c \x -\nabla h(\x) = 2\z$.

Our goal is now to find constraints that assure that this system of equations have only one sparse solution. Before getting into the details we outline the overall idea.
For simplicity consider two differentiable strictly convex functions $\tilde{h}$ and $\tilde{g}$.
Their sum is minimized by the stationary point $\x_s$ fulfilling $- \nabla \tilde{h}(\x_s) = \nabla \tilde{g}(\x_s)$. 
Since $\tilde{g}$ is strictly convex its directional derivative $\skal{\nabla \tilde{g} (\x_s + t \vv), \vv}$ is increasing for all directions $\vv \neq 0$. 
Similarly $\skal{-\nabla \tilde{h} (\x_s + t \vv), \vv}$ is decreasing for all $\vv\neq 0$ since $-\tilde{h}$ is strictly concave. Therefore $\skal{-\nabla \tilde{h} (\x_s + t \vv), \vv} < \skal{\nabla \tilde{g} (\x_s + t \vv), \vv}$ which means that $\x_s$ is the only stationary point.
In what follows we will estimate the growth of the directional derivatives of the functions involved in \eqref{eq:vectorobj} in order to show a similar contradiction.
For the function $h$ we do not have convexity, however due to \eqref{eq:vectorRIP} we shall see that it behaves essentially like a convex function for sparse vectors $\x$. Additionally, because of the non-convex perturbation $-\delta_c \|\x\|^2$ we need somewhat sharper estimates than just growth of the directional derivatives of $g$.

We fist consider the estimate for the derivatives of $h$. Note that $
\nabla h(\x)=2 \delta_c \x + 2(A^T A - I) \x -2 A^T b$,
and therefore
\begin{equation*}
\skal{\nabla h(\x+\vv)- \nabla h(\x), \vv} = 2\delta_c\|\vv\|^2+2\left(\|A \vv\|^2- \|\vv\|^2\right).
\end{equation*}
Applying \eqref{eq:vectorRIP} now shows that
\begin{equation}
\skal{\nabla h(\x+\vv)-\nabla h(\x), \vv} \ge 2\delta_c\|\vv\|^2-2\delta_c\|\vv\|^2 = 0,
\label{eq:hbnd}
\end{equation}
when $\supp(\vv)\leq c$.

Next we need a similar bound on the sub-gradients of $g$. 
In order to guarantee uniqueness of a sparse stationary point we need to show that they grow faster than $2\delta_c \|\vv\|^2$.
The following three lemmas show that provided that the vector $\z$, where $2\z \in \partial g (\x)$ is the sub-gradient, has elements that are not too close to the thresholds $\pm \sqrt{\mu}$ this will be true.

\begin{lemma}\label{lemma:bnd1}
Assume that $2\z \in \partial g(\x)$.
If 
\begin{equation}
\left|
z_i
\right| > \frac{\sqrt{\mu}}{1- \delta_c}
\label{eq:subgradbnd1}
\end{equation}
then for any $\z'$ with $2\z' \in \partial g(\x+\vv)$ we have
\begin{equation}
z_i' -z_i >  \delta_c v_i \quad \text{ if } v_i > 0
\end{equation}
and
\begin{equation}
z_i' -z_i <  \delta_c v_i \quad \text{ if } v_i < 0.
\end{equation}
\end{lemma}
\begin{proof}
We first assume that $x_i > 0$. Because of \eqref{eq:subgradbnd1} and \eqref{eq:vectorsubgrad} we have $x_i = z_i > \frac{\sqrt{\mu}}{1-\delta_c}$.
There are now two possibilities:
\begin{itemize}
\item If $v_i > 0$ then $x_i+v_i > \sqrt{\mu}$ and by \eqref{eq:vectorsubgrad} we therefore must have that $z'_i = x_i+v_i = z_i + v_i > z_i + \delta_c v_i$.
\item If $v_i < 0$ we consider the line 
\begin{equation}
l(x) = 2z_i + 2 \delta_c (x-x_i).
\end{equation}
See the left graph of Figure~\ref{fig:linefig1}. We will show that this line is an upper bound on the sub-gradients for all $v_i < 0$.
\begin{figure*}[htb]
\begin{center}
\includegraphics[width=70mm]{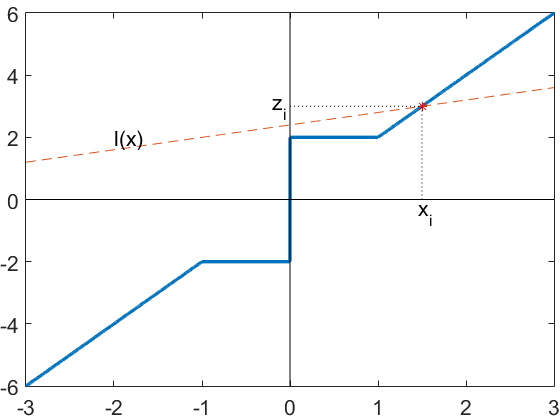}
\includegraphics[width=70mm]{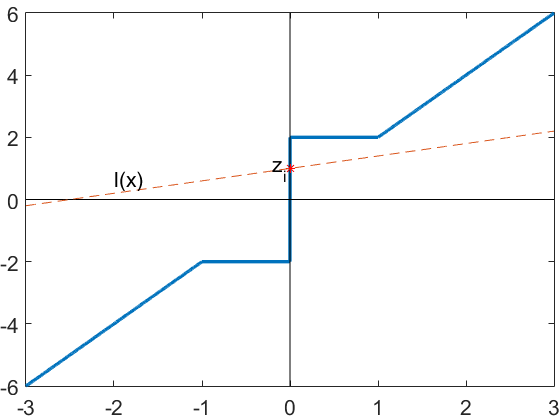}
\end{center}
\caption{Illustration of the subdifferential $\partial g(x)$ and the line $l(x)$, when ~\eqref{eq:subgradbnd1} holds (left) and when \eqref{eq:subgradbnd2} holds (right). }
\label{fig:linefig1}
\end{figure*}

We note that for $x < x_i$ we have
\begin{equation}
l(x) = \underbrace{2z_i-2 x_i}_{=0}+\underbrace{2(1-\delta_c) x_i}_{> 2(1-\delta_c)x} +2 \delta_c x > 2x.
\end{equation}
Furthermore
\begin{equation}
l(x) = 2(1-\delta_c)x_i + 2 \delta_c x > 2\sqrt{\mu}+2\delta_c x.
\end{equation}
The right hand side is clearly larger than both $2\sqrt{\mu}$ for $x \geq 0$. For $-\sqrt{\mu} \leq  x \leq 0$ we have $2\sqrt{\mu}+2\delta_c x > 2\sqrt{\mu}+2x \geq 0 \geq -2\sqrt{\mu}$. This shows that the line $l(x)$ is an upper bound on the subgradients of $g$ 
for every $x < x_i$, that is $l(x_i+v_i) > 2z_i'$ for all $v_i < 0$ and since $l(x_i+v_i)=2z_i+2\delta_c v_i$ we get $2z'_i < 2z_i + 2 \delta_c v_i$.
\end{itemize}
The proof for the case $x_i < 0$ is similar.
\end{proof}

\begin{lemma}\label{lemma:bnd2}
Assume that $2\z \in \partial g(\x)$.
If 
\begin{equation}
\left|
z_i
\right| < (1- \delta_c)\sqrt{\mu}
\label{eq:subgradbnd2}
\end{equation}
then for any $\z'$ with $2\z' \in \partial g(\x+\vv)$ we have
\begin{equation}
z_i' -z_i >  \delta_c v_i \quad \text{ if } v_i > 0
\end{equation}
and
\begin{equation}
z_i' -z_i <  \delta_c v_i \quad \text{ if } v_i < 0.
\end{equation}
\end{lemma}
\begin{proof}
By \eqref{eq:subgradbnd2} we see that $x_i=0$. We first assume that $v_i > 0$ and consider the line $l(x) = 2z_i + 2\delta_c x$, see the right graph of Figure~\ref{fig:linefig1}.
We have that
\begin{equation}
l(x) < 2(1-\delta_c)\sqrt{\mu} + 2\delta_c x. 
\end{equation}
The right hand side is less than $2(1-\delta_c)\sqrt{\mu} + 2\delta_c \sqrt{\mu} = 2\sqrt{\mu}$ when $0 < x \leq \sqrt{\mu}$ and less than $2(1-\delta_c)x + 2\delta_c x = 2x$ when $x > \sqrt{\mu}$.
Therefore $l(x)$ is a lower bound on the subgradients of $g$ for all $x>0$ which gives
$l(v_i) < 2z'_i$ for $v_i > 0$ and since $l(v_i) = 2z_i+2 \delta_c v_i$ we get $2z_i' > 2z_i + 2\delta_c v_i$. 
The case $v_i < 0$ is similar.
\end{proof}

\begin{lemma}\label{lemma:subgradbnd}
Assume that $2\z \in \partial g(\x)$. If the elements $z_i$ fulfill
$|z_i| \notin [(1-\delta_c)\sqrt{\mu},\frac{\sqrt{\mu}}{1-\delta_c}]$ for every $i$, then for any $\z'$ with $2\z' \in \partial g(\x+\vv)$ we have
\begin{equation}
\skal{\z'-\z,\vv} >  \delta_c \|\vv\|^2,
\end{equation}
as long as $\|\vv\| \neq 0$.
\end{lemma}

\begin{proof}
The is an immediate consequence of the previous two results.
We have according to Lemmas \ref{lemma:bnd1} and \ref{lemma:bnd2} that
\begin{equation}
(z'_i-z_i)v_i >  \delta_c v_i^2,
\end{equation}
for all $i$ with $v_i \neq 0$. Since $v_i=0$ gives $(z'_i-z_i)v_i=0$
summing over $i$ gives
\begin{equation}
\skal{\z'-\z,\vv} >  \delta_c \|\vv\|^2,
\end{equation}
as long as $\|\vv\| \neq 0$.
\end{proof}

We are now ready to consider the distribution of stationary points. Set $\z=(I-A^T A)\x_s + A^T \bb$ and recall that $2\z \in \partial g(\x_s)$ for stationary points $\x_s$ (Lemma \ref{lemma:lowrankstatpt}).

\begin{theorem}\label{thm:statpoint}
Assume that $\x_s$ is a stationary point of $f$ and that each element $z_i$ fulfills $|z_i| \notin [(1-\delta_c)\sqrt{\mu},\frac{\sqrt{\mu}}{1-\delta_c}]$.
If $\y_s$ is another stationary point of $f$ then $\supp(\y_s-\x_s) > c$.
\end{theorem}
\begin{proof}
Assume that $\supp(\y_s-\x_s) \leq c$.
We first note that
\begin{equation}
2\delta_c\x - \nabla h(\x) = 2(I -A^T A) \x + 2A^T \bb.
\end{equation}
Since $\x_s$ and $\y_s$ are both stationary points we have
%\begin{eqnarray}
$2\delta_c\x_s - \nabla h(\x_s) = 2\z$ and % \\
$2\delta_c\y_s - \nabla h(\y_s) = 2\z'$,
%\end{eqnarray}
where $2\z \in \partial g(\x_s)$ and $2\z' \in \partial g(\y_s)$.
Taking the difference between the two equations gives
\begin{equation}
2\delta(\y_s-\x_s) - \left(\nabla h(\y_s)-\nabla h(\x_s)\right) = 2(\z'-\z),
\end{equation}
which implies
\begin{equation}
2\delta_c\|\vv\|^2 - \skal{\nabla h(\x + \vv)-\nabla h(\x), \vv} = 2\skal{\z'-\z,\vv},
\end{equation}
where $\vv = \y_s-\x_s$.
However, according to \eqref{eq:hbnd} the left hand side is less than $2\delta\|\vv\|^2$
if $\supp(\vv) \leq c$ which contradicts Lemma \ref{lemma:subgradbnd}.
\end{proof}
\paragraph{A one dimensional example.}
We conclude this section with a simple one dimensional example which shows that the bounds \eqref{eq:subgradbnd1} and \eqref{eq:subgradbnd2} cannot be made sharper.
Figure~\ref{fig:simpleex} shows the function $r_1(x)+(\frac{1}{\sqrt{2}}x-b)^2$ for different values of $b \geq 0$.
It is not difficult to verify that this function can have three stationary points (when $b \geq 0$).
The point $x=0$ is stationary if $b \leq \sqrt{2}$, $x=2-\sqrt{2}b$ if $\frac{1}{\sqrt{2}} < b <\sqrt{2}$ and $x=\sqrt{2}b$ if $b \geq \frac{1}{\sqrt{2}}$, see Figure~\ref{fig:simpleex}.
For this example $A=\frac{1}{\sqrt{2}}$ and therefore $(1-\delta)|x|^2 \leq |Ax|^2 \leq (1+\delta)|x|^2$ holds with $1-\delta = \frac{1}{2}$.

Now suppose that $b \leq \sqrt{2}$ and that we, using some algorithm, find the stationary point $x=0$. We then have
\begin{equation}
z = (1-A^T A)x+A^T b = \frac{1}{\sqrt{2}} b.
\end{equation}
Theorem 3.3 now tells us that $x=0$ is the unique stationary point if 
\begin{equation}
\frac{1}{\sqrt{2}}b \notin \left[1-\delta,\frac{1}{1-\delta} \right] \Leftrightarrow b \notin \left[\frac{1}{\sqrt{2}}, 2\sqrt{2} \right].
\end{equation}
Note that the lower interval bound $b < \frac{1}{\sqrt{2}}$ is precisely when $x=0$ is unique, see the leftmost graph in  Figure~\ref{fig:simpleex}.

Similarly suppose that $b\geq \frac{1}{\sqrt{2}}$. For the point $x=\sqrt{2}b$ we get
\begin{equation}
\frac{1}{2}z = (1-A^T A)x+A^T b = \frac{1}{2}\sqrt{2}b+\frac{1}{\sqrt{2}} b = \sqrt{2} b.
\end{equation}
Theorem 3.3 now shows that $x=\sqrt{2}b$ is unique if
\begin{equation}
\sqrt{2} b \notin \left[1-\delta,\frac{1}{1-\delta} \right] \Leftrightarrow b \notin \left[\frac{1}{2\sqrt{2}}, \sqrt{2} \right].
\end{equation}
Here the upper interval bound $b>\sqrt{2}$ is precisely when $x=\sqrt{2}b$ is unique, see rightmost graph in Figure~\ref{fig:simpleex}.
Hence for this example Theorem 3.3 is tight in the sense that it would be able to verify uniqueness of the stationary point for every $b$ where this holds.
\begin{figure*}
\begin{center}
\def\figwidths{31mm}
\begin{tabular}{ccccc}
\includegraphics[width=\figwidths]{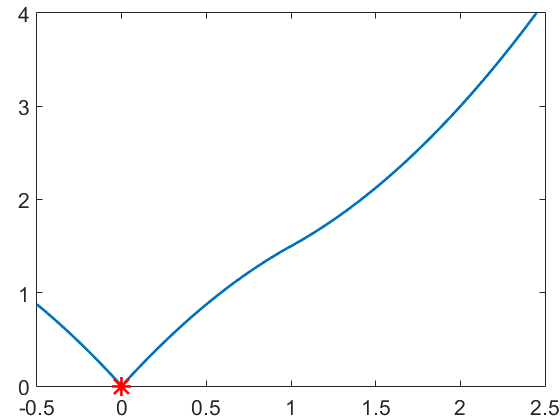}&
\includegraphics[width=\figwidths]{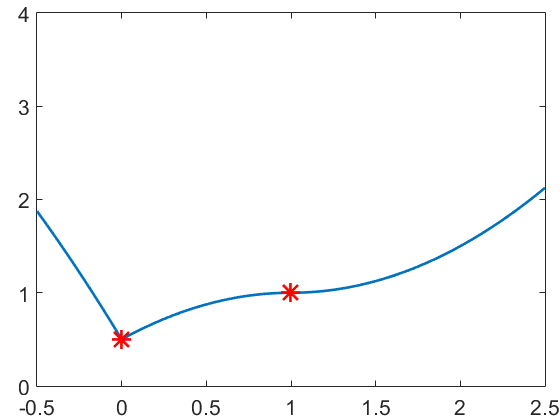}&
\includegraphics[width=\figwidths]{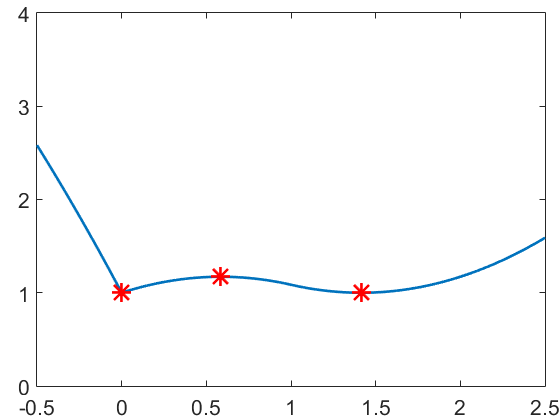}&
\includegraphics[width=\figwidths]{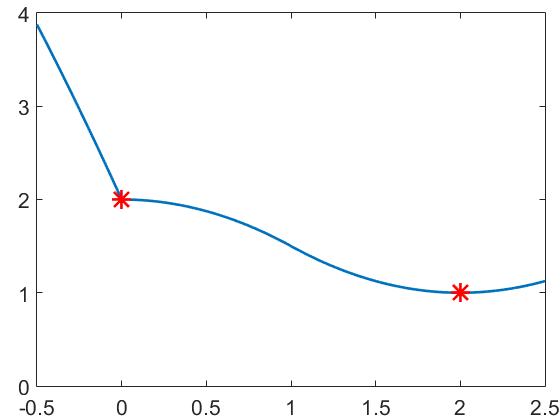}&
\includegraphics[width=\figwidths]{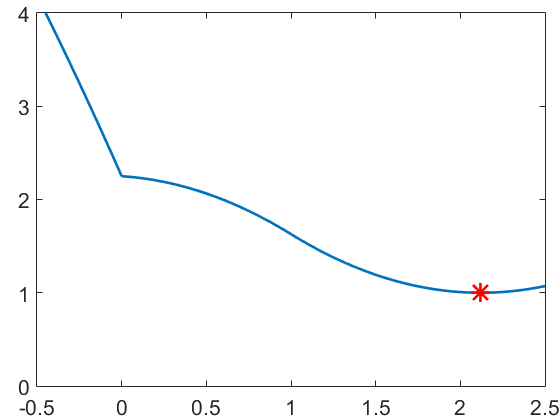}\\
$b = 0$ & $b = \frac{1}{\sqrt{2}}$ & $b=1$ & $b = \sqrt{2}$ & $b=1.5$
\end{tabular}
\end{center}
\caption{The function $r_1(x)+(\frac{1}{\sqrt{2}}x-b)^2$ and its stationary points (red) for different values of $b$. When $b$ is close to the threshold $1$ the function has multiple stationary points.}
\label{fig:simpleex}
\end{figure*}

\section{Low Rank Regularization}\label{sec:rank}
Next we generalize the vector formulation from the previous section to a matrix setting.
We let 
\begin{equation}
F(X) = r_\mu (\bfsigma(X)) + \|\A X - \bb\|_F^2, 
\end{equation}
and assume that $\A$ is a linear operator $\mathbb{R}^{m\times n} \rightarrow \mathbb{R}^k$ fulfilling \eqref{eq:matrisRIP} for all $X$ with $\rank(X) \leq r$. As in the vector case we can (ignoring constants) equivalently write
\begin{equation}
F(X) = G(X) - \delta_r\|X\|_F^2 + H(X),
\end{equation}
where 
$H(X) = \delta_r \|X\|_F^2 + \left(\|\A X\|^2-\|X\|_F^2\right)-2\skal{X, \A^* \bb}$
and 
$G(X) = g(\bfsigma(X))$,
with $g$ as in \eqref{eq:gdef}. 
The first estimate
\begin{equation}
\skal{\nabla H(X+V)- \nabla H(V), V} \geq 0
\label{eq:hbnd_rank}
\end{equation}
follows directly from \eqref{eq:matrisRIP} if $\rank(V) \leq r$. Our goal is now to show a matrix version of Lemma \ref{lemma:subgradbnd}. 

\begin{lemma}\label{UVlemma}
Let $\x$,$\x'$,$\z$,$\z'$ be fixed vectors with non-increasing and non-negative elements 
such that $\x \neq \x'$, $2\z \in \partial g(\x)$ and $2\z' \in \partial g(\x')$.
Define $X' = U' D_{\x'} V'^T$, $X = U D_\x V^T$, $Z' = U' D_{\z'} V'^T$, and  $Z = U D_\z V^T$ as functions of unknown orthogonal matrices $U$, $V$, $U'$ and $V'$. 
If
\begin{equation}
a^* = \min_{U,V,U',V'}\frac{\skal{Z'-Z,X'-X}}{\|X'-X\|_F^2} \leq 1
\label{eq:matrixfraction}
\end{equation}
then
\begin{equation}
a^* = \min _{\pi} \frac{\skal{M_{\pi,\e} \z' - \z,M_{\pi,\e} \x' - \x}}{\|M_{\pi,\e} \x'-\x\|^2},
\end{equation}
where $\e$ is a vector of all ones.
\end{lemma}

\begin{proof}
We may assume that $U = I_{m\times m}$ and $V=I_{n\times n}$.
We first note that $(U',V')$ is a minimizer of \eqref{eq:matrixfraction} if and only if
\begin{equation}
\skal{Z'-Z,X'-X} \leq a^* \|X'-X\|_F^2.
\end{equation}
This constraint can equivalently be written
\begin{equation}
C - \skal{Z'-a^*X',X} - \skal{Z-a^* X, X'} \leq 0,
\end{equation}
where $C=\skal{Z',X'}+\skal{Z,X}-a^*(\|X'\|_F^2+\|X\|_F^2)$ is independent of $U'$ and $V'$.
Thus any minimizer of \eqref{eq:matrixfraction} must also maximize 
\begin{equation}
\skal{U'D_{\z'-a^*\x'}V'^T,D_\x} + \skal{D_{\z-a^* \x}, U'D_{\x'} V'^T}.
\label{eq:diagmatrices}
\end{equation}
For ease of notation we now assume that $m \leq n$, that is, 
the number of rows are less than the columns (the opposite case can be handled by transposing).
Equation \eqref{eq:diagmatrices} can now be written
\begin{equation}
\x^T M(\z'-a^*\x')+ (\z-a^*\x)^T M\x',
\label{eq:linearmatrixobjective}
\end{equation}
where $M = U' \odot V'_{1,1}$, $V'_{1,1}$ is the upper left $m\times m$ block of $V'$ and $\odot$ denotes element wise multiplication.
Since both $U'$ and $V'$ are orthogonal it is easily shown (using the Cauchy-Schwartz inequality) that $M$ is DSS.

Note that objective \eqref{eq:linearmatrixobjective} is linear in $M$ and therefore optimization over the set of DSS 
matrices is guaranteed to have an extreme point $M_{\pi,\vv}$ that is optimal. 
Furthermore, since $a^* \leq 1$ the vectors $\x$,$\x'$,$\z-a^*\x$ and $\z'-a^*\x'$ all have positive entries, 
and therefore the maximizing matrix has to be $M_{\pi,\e}$ for some permutation $\pi$.
Since $M_{\pi,\e}$ is orthogonal and $ M_{\pi,\e}= M_{\pi,\e}\odot M_{\pi,\e}$,
$U' = M_{\pi,\e}$ and $V'_{1,1} =  M_{\pi,\e}$ will be optimal when maximizing \eqref{eq:linearmatrixobjective} over $U'$ and $V'_{1,1}$.
An optimal $V'$ in \eqref{eq:diagmatrices} can now be chosen to be 
%\begin{equation}
$V' = \begin{bmatrix}
M_{\pi,\e} & 0 \\
0 & I
\end{bmatrix}$.
%\end{equation}
Note this choice is somewhat arbitrary since only the upper left block of $V'$ affects the value of \eqref{eq:diagmatrices}.
The matrices $U' Z' V'^T$ and $U' X' V'^T$ are now diagonal, with diagonal elements $M_{\pi,\e}\z'$ and $M_{\pi,\e}\x'$, which concludes the proof.
\end{proof}

\begin{corollary}
Assume that $2Z \in \partial G(X)$.
If the singular values of the matrix $Z$ fulfill $z_i \notin [(1- \delta_r)\sqrt{\mu}, \frac{\sqrt{\mu}}{1- \delta_r}]$,
 then for any $2Z' \in \partial G(X')$ we have 
\begin{equation}
\skal{Z'-Z,X'-X} >  \delta_r \|X'-X\|_F^2,
\end{equation}
as long as $\|X'-X\|_F \neq 0$.
\end{corollary}
\begin{proof}
We will first prove the result under the assumption that $\bfsigma(X) \neq \bfsigma(X')$ and then generalize to the general case using a continuity argument. 
For this purpose we need to extend the infeasible interval somewhat.
Since $\delta_r < 1$ and the complement of $[(1- \delta_r)\sqrt{\mu}, \frac{\sqrt{\mu}}{1- \delta_r}]$ is open there is an 
$\epsilon > 0$ such that $z_i \notin [(1- \delta_r-\epsilon)\sqrt{\mu}, \frac{\sqrt{\mu}}{1- \delta_r-\epsilon}]$ and $\delta_r+\epsilon < 1$.
Now assume that $a^* > 1$ in \eqref{eq:matrixfraction}, then clearly
\begin{equation}
\skal{Z'-Z,X'-X} %>  \|X'-X\|_F^2 
> (\delta_r+\epsilon) \|X'-X\|_F^2,
\label{eq:epsilonbound}
\end{equation}
since $\delta_r+\epsilon < 1$. 
Otherwise $a^* \leq 1$ and we have
\begin{equation}
\frac{\skal{Z'-Z,X'-X}}{\|X'-X\|_F^2} \geq  \frac{\skal{M_{\pi,\e} \z' - \z,M_{\pi,\e} \bfsigma(X') - \bfsigma(X)}}{\|M_{\pi,\e} \bfsigma(X') - \bfsigma(X)\|^2}.
\end{equation}
According to Lemma \ref{lemma:subgradbnd} the right hand side is strictly larger than $\delta_r+\epsilon$, 
which proves that \eqref{eq:epsilonbound} holds for all $X'$ with $\bfsigma(X')\neq \bfsigma(X)$.

For the case $\bfsigma(X') = \bfsigma(X)$ and $\|X'-X\|_F \neq 0$ we will now prove that
\begin{equation}
\skal{Z'-Z,X'-X}
\geq (\delta_r+\epsilon) \|X'-X\|_F^2,
\label{eq:epsilonbound2}
\end{equation}
using continuity of the scalar product and the Frobenius norm. 
Since $\epsilon > 0$ this will prove the result.

We must have $\sigma_1(X)>0$ since otherwise $X$ = $X'$ = 0 and therefore $\|X'-X\|_F = 0$.
By the definition of the sub differential we therefore know that $z_1 \geq \sqrt{\mu}$ and by 
the assumptions of the lemma we have that $z_1 > \frac{\sqrt{\mu}}{1- \delta_r-\epsilon}$.

If $X=UD_{\bfsigma(X)}V^T$ we now define $\bar{X}(t) = UD_{\bfsigma(\bar{X}(t))}V^T$, where
\begin{equation}
\sigma_i(\bar{X}(t)) = 
\begin{cases}
\sigma_1(X)+t & \text{if } i=1 \\ 
\sigma_i(X) & \text{otherwise}
\end{cases}.
\end{equation}
Similarly we define $\bar{Z}(t) = UD_{\bar{\z}(t)}V^T$, where
\begin{equation}
\bar{z}_i(t) = 
\begin{cases}
z_1+t & \text{if } i=1 \\ 
z_i & \text{otherwise}
\end{cases}.
\end{equation}
It is now clear that $2\bar{Z}(t) \in \partial G(\bar{X}(t))$ and $\bar{z}_i(t) \notin [(1- \delta_r-\epsilon)\sqrt{\mu}, \frac{\sqrt{\mu}}{1- \delta_r-\epsilon}]$,
for all $t \geq 0$. Further more $\bar{Z}(t) \rightarrow \bar{Z}(0) = Z$ and $\bar{X}(t) \rightarrow \bar{X}(0) = X$ when $t \rightarrow 0^+$.
Since $\bfsigma(\bar{X}(t))\neq \bfsigma(X')$ for $t > 0$ we have by \eqref{eq:epsilonbound2} that
\begin{equation}
\skal{Z'-\bar{Z}(t),X'-\bar{X}(t)} > (\delta_r+\epsilon) \|X'-\bar{X}(t)\|_F^2,
\end{equation}
for all $t > 0$.
By continuity of the Frobenius norm and the scalar product we can now conclude that \eqref{eq:epsilonbound2} holds.
\end{proof}

\begin{theorem}
Assume that $X_s$ is a stationary point of $F$, that is, $(I-\A^*\A)X_s + \A^* \bb = Z$,  where $2Z \in \partial G(X)$ and the singular values of $Z$ fulfill $\sigma_i(Z) \notin [(1- \delta_r)\sqrt{\mu}, \frac{\sqrt{\mu}}{1- \delta_r}]$. 
If $X'_s$ is another stationary point then $\text{rank}(X'_s-X_s) > r$.
\end{theorem}
The proof is similar to that of Theorem \ref{thm:statpoint} and therefore we omit it.

\section{Experiments}

In this section we evaluate the proposed methods on a few synthetic experiments.
For low rank recovery we compare the two formulations 
\begin{eqnarray}
\mu' \|X\|_* + \|\A X -\bb\|^2 \label{eq:nuclear} \\
r_{\mu}(\bfsigma(X)) + \|\A X -\bb\|^2\label{eq:our}
\end{eqnarray}
for low rank recovery for varying regularization strengths $\mu$ and $\mu'$. 
Note that the proximal operator of the nuclear norm, $\argmin_X \mu' \|X\|_*+\|X-W\|^2$, performs soft thresholding at $\frac{\mu'}{2}$ while that of $r_\mu$,
$\argmin_X \mu r_\mu(\bfsigma(X))+\|X-W\|^2$, thresholds at $\sqrt{\mu}$ \cite{larsson-olsson-ijcv-2016}.
In order for the methods to roughly suppress an equal amount of noise we therefore use $\mu' = 2 \sqrt{\mu}$ in \eqref{eq:nuclear}.

For sparse recovery we compare the two formulations 
\begin{eqnarray}
\mu' \|\x\|_1 + \|A \x -\bb\|^2. \label{eq:l1} \\
r_{\mu}(\x) + \|A \x -\bb\|^2\label{eq:ourvector}
\end{eqnarray}
for varying regularization strengths $\mu$ and $\mu'$. 
Similarly to the rank case, the proximal operator of the $\ell_1$-norm, $\argmin_\x \mu' \|\x \|_1+\|\x-\z\|^2$, performs soft thresholding at $\frac{\mu'}{2}$ while that of $r_\mu$,
$\argmin_\x \mu r_\mu(x)+\|\x-\z \|^2$, thresholds at $\sqrt{\mu}$ \cite{larsson-olsson-ijcv-2016}.
We therefore use $\mu' = 2 \sqrt{\mu}$ in \eqref{eq:l1}.

\subsection{Optimization Method}
Because of its simplicity we use the GIST approach from \cite{gong-etal-2013}.
Given a current iterate $X_k$ this method uses a trust region formulation that approximates the data term $\|\A X - \bb\|^2$ with the linear function $2\skal{ \A^* \A X_k - \A^* \bb, X }$.
In each step the algorithm therefore finds $X_{k+1}$ by solving
\begin{equation}
\min_X r_\mu(\bfsigma(X))+2\skal{ \A^* \A X_k - \A^* \bb, X }  + \tau_k \|X-X_k\|^2.
\end{equation}
Here the third term $\tau_k \|X-X_k\|^2$ restricts the search to a neighborhood around $X_k$.
Completing squares shows that the above problem is equivalent to
\begin{equation}
%r_\mu(\bfsigma(X)) + \tau_k \left\|X- \underbrace{\left(X_k -\frac{1}{\tau_k}(\A^* \A X_k- A^* \bb)  \right)}_{:= M} \right\|^2.
\min_X r_\mu(\bfsigma(X)) + \tau_k \left\|X- M \right\|^2,
\label{eq:tauapprox}
\end{equation}
where $M =X_k -\frac{1}{\tau_k}(\A^* \A X_k- A^* \bb) $.
Note that if $\tau_k=1$ then any fixed point of \eqref{eq:tauapprox} is a stationary point by Lemma \ref{lemma:lowrankstatpt}.
The optimization of \eqref{eq:tauapprox} will be separable in the singular values of $X$.
For each $i$ we minimize $-\max(\sqrt{\mu}-\sigma_i(X),0)^2+\tau_k(\sigma_i(X)-\sigma_i(M))^2$.
Since singular values are always positive there are three possible minimizers:
 $\sigma_i(X) =\sigma_i(M)$, $\sigma_i(X) = 0$ and $\sigma_i(X) = \frac{\tau_k\sigma_i(M)-\sqrt{\mu}}{\tau_k-1}$. 
In our implementation we simply test which one of these yields the smallest objective value.
(If $\tau_k=1$ it is enough to test $\sigma_i(X) =\sigma_i(M)$ and $\sigma_i(X) = 0$.)
For initialization we use $X_0=0$.

In summary our algorithm consists of repeatedly solving \eqref{eq:tauapprox} for a sequence of $\{\tau_k\}$. In the experiments we noted that using $\tau_k = 1$ for all $k$ sometimes resulted
in divergence of the method due to large step sizes.
We therefore start from a larger value ($\tau_0 = 5$ in our implementation) and reduce towards $1$ as long as this results in decreasing objective values. Specifically we set $\tau_{k+1} = \frac{\tau_k-1}{1.1} + 1$ if the previous step was successful in reducing the objective value. Otherwise we increase $\tau$ according to $\tau_{k+1} = 1.5(\tau_k-1) + 1$.

The sparsity version of the algorithm is almost identical to the one described above. In each step we find $\x_{k+1}$ by minimizing
\begin{equation}
r_\mu(\x) + \tau_k \left\|\x- \underbrace{\left(\x_k -\frac{1}{\tau_k}(A^T A \x_k- A^T \bb)  \right)}_{:= m} \right\|^2.
\label{eq:tauapprox}
\end{equation}
The optimization is separable and for each element $x_i$ we minimize $-\max(\sqrt{\mu}-|x_i|,0)^2+\tau_k(x_i-m_i)^2$.
It is easy to show that there are four possible choices $x_i = m_i$, $x_i = \frac{\tau_k m_i\pm\sqrt{\mu}}{\tau_k-1}$ and $x_i=0$ that can be optimal. 
In our implementation we simply test which one of these yields the smallest objective value.
(If $\tau_k=1$ it is enough to test $x_i=m_i$ and $x_i=0$.)
For initialization we use $\x_0=0$.

\begin{figure*}
\def\w{45mm}
\begin{center}
\includegraphics[width=\w]{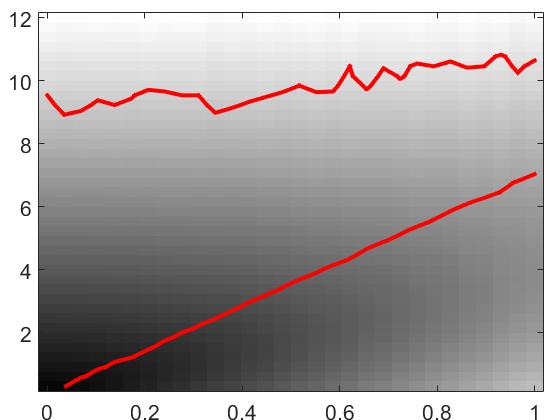} 
\includegraphics[width=\w]{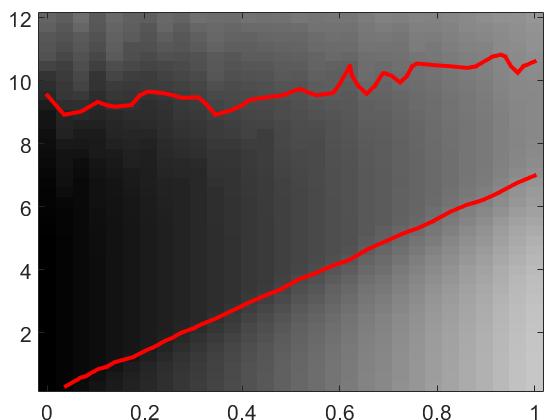} 
\includegraphics[width=\w]{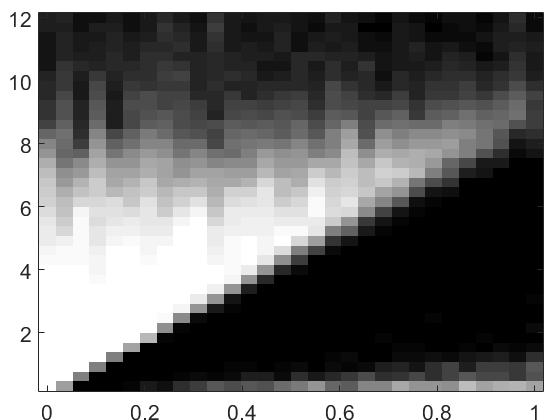}\\
\end{center}
\caption{Low rank recovery results for varying noise level (x-axis) and regularization strength (y-axis) with random $400 \times 400$ $\A$ with $\delta=0.2$.
\emph{Left}: Average distances between \eqref{eq:nuclear} and the ground truth for $\mu$ between $0$ and $12$.
(red curves marks the area where the obtained solution has $\rank(X)=5$).
\emph{Middle}: Average distances between \eqref{eq:our} the ground truth.
\emph{Right}: Number of instances where \eqref{eq:our} could be verified to be optimal for $\delta=0.2$ (white = all, black = none).}
\label{fig:rank_result}
\end{figure*}

\subsection{Low Rank Recovery}
\begin{figure}
\begin{center}
\def\w{41mm}
\includegraphics[width=\w]{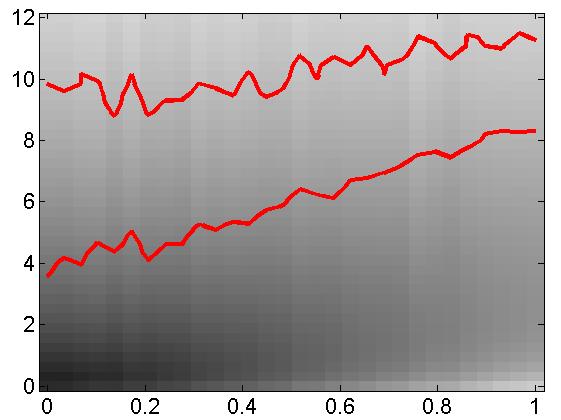} 
\includegraphics[width=\w]{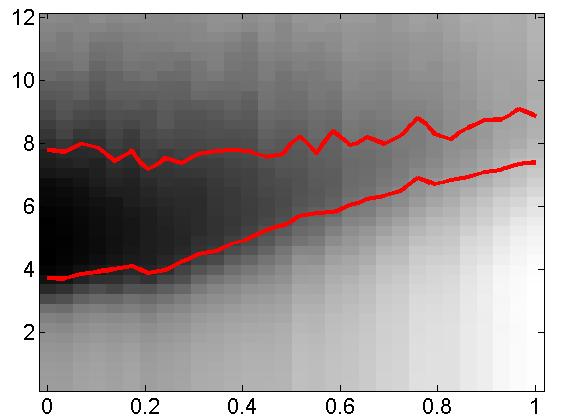} 
\end{center}
\caption{Low rank recovery results varying noise level (x-axis) and regularization strength (y-axis) with random $300 \times 400$ $\A$ (and unknown $\delta$).
\emph{Left}: Average distances between \eqref{eq:nuclear} and the ground truth.
(red curves marks the area where the obtained solution has $\rank(X)=5$).
\emph{Middle}: Average distances between \eqref{eq:our} the ground truth.}
\label{fig:lowrankunknowndelta}
\end{figure}
In this section we test the proposed method on synthetic data.
We generate $20 \times 20$ ground truth matrices of rank $5$ by randomly selecting $20\times 5$ matrices $U$ and $V$ with $\mathcal{N}(0,1)$ elements and multiplying $X=UV^T$. By column stacking matrices the linear mapping $\A : \mathbb{R}^{m\times n} \rightarrow \mathbb{R}^p$ can be represented by a $p \times mn$ matrix $\hat{A}$.
For a given rank $r < \min(m,n)$ it is a difficult problem to determine the exact $\delta_r$ for which \eqref{eq:matrisRIP} holds. However if we consider unrestricted solutions ($r = \min(m,n)$)  \eqref{eq:matrisRIP} reduces to a singular value bound.
It is easy to see that if $p \geq mn$ and $\sqrt{1-\delta_r} \leq \sigma_i(\hat{A}) \leq \sqrt{1-\delta_r}$ for all $i$ then \eqref{eq:matrisRIP} clearly holds for all $X$.
For under-determined systems finding the value of $\delta_r$ is much more difficult. However for a number of random matrix families it can be proven that \eqref{eq:matrisRIP} will hold with high probability when the matrix size tends to infinity. For example \cite{candes-tao-2006,recht-etal-siam-2010} mentions random matrices with Gaussian entries, Fourier ensembles, random projections and matrices with Bernoulli distributed elements.

For Figure~\ref{fig:rank_result} we randomly generated problem instances for low rank recovery. Each instance uses a matrix $\hat{A}$ of size $20^2 \times 20^2$ with $\delta = 0.2$ which was generated by first randomly sampling the elements of a matrix $\tilde{A}$ a Gaussian $\mathcal{N}(0,1)$ distribution. The matrix $\hat{A}$ was then constructed from $\tilde{A}$ by modifying the singular values to be evenly distributed between $\sqrt{1-\delta}$ and $\sqrt{1+\delta}$.
To generate a ground truth solution and a $\bb$ vector we then computed $\bb=\A X+\epsilon$, where all elements of $\epsilon$ are $\mathcal{N}(0,\sigma^2)$.
We then solved \eqref{eq:nuclear} and \eqref{eq:our} for varying noise level $\sigma$ and regularization strength $\mu$
and computed the distance between the obtained solution and the ground truth.

The averaged results (over 50 random instances for each $(\sigma,\mu)$ setting) are shown in
Figure~\ref{fig:rank_result}.
(Here black means low and white means a high errors. Note that the color-maps of left and middle image are the same. The red curves show the area where the computed solution has the correct rank.)
From Figure~\ref{fig:rank_result} it is quite clear that the nuclear norm suffers from a shrinking bias.
It consistently gives the best agreement with the ground truth data for values of $\mu$ that are not big enough to generate low rank. The effect becomes more visible as the noise level increases since a larger $\mu$ is required to suppress the noise.
In contrast, \eqref{eq:rankrelax} gives the best fit at the correct rank for all noise levels. This fit was consistently better than that of \eqref{eq:nuclear} for all noise levels.
To the right in Figure~\ref{fig:rank_result} we show the fraction of problem instances that could be verified to be optimal 
(by computing $Z = (I - \A^* \A)X_s + \A^* \bb$ and checking that $\sigma_i(Z) \notin [(1-\delta)\sqrt{\mu},\frac{\sqrt{\mu}}{1-\delta}]$).
It is not unexpected that verification works best when the noise level is moderate and a solution with the correct rank has been recovered.
In such cases the recovered $Z$ is likely to be close to low rank.
Note for example that in the noise free case, that is, $\bb = \A X_0$ for some low rank $X_0$ then
$Z=(I-\A^* \A)X_0 + \A^*\A X_0 = X_0$.

In Figure~\ref{fig:lowrankunknowndelta} we randomly generated under-determined problems with $\hat{A}$ of size $300 \times 20^2$ with Gaussian $\mathcal{N}(0,\frac{1}{300})$ elements and $\bb$ vector as described previously.
Even though we could not verify the optimality of the obtained solution (since $\delta$ is unknown) our approach consistently outperformed nuclear norm regularization which exhibits the same tendency to achieve a better fit for non-sparse solutions.
In this setting \eqref{eq:nuclear} performed quite poorly, failing to simultaneously achieve a good fit and a correct rank (even for low noise levels).

\subsection{Non-rigid Reconstruction}
\begin{figure*}
\begin{center}
\def\w{9mm}
\begin{tabular}{|c|c|c|c|}
\hline
\includegraphics[width=\w]{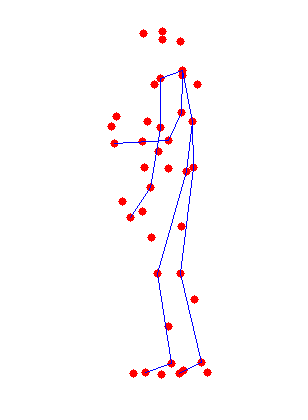}
\includegraphics[width=\w]{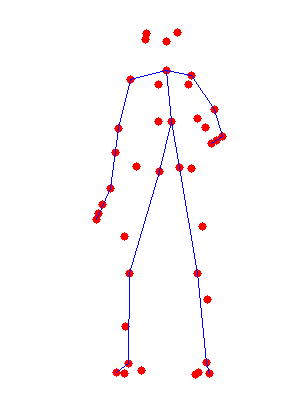}
\includegraphics[width=\w]{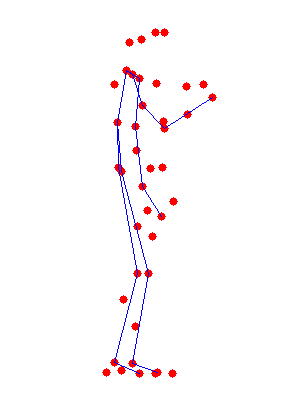}
\includegraphics[width=\w]{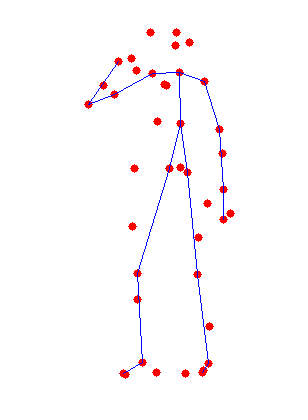}&
\includegraphics[width=\w]{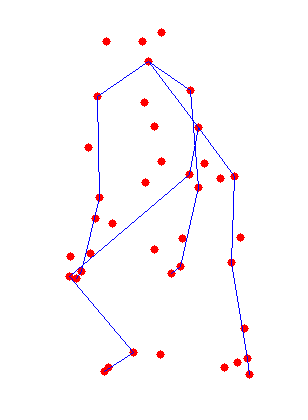}
\includegraphics[width=\w]{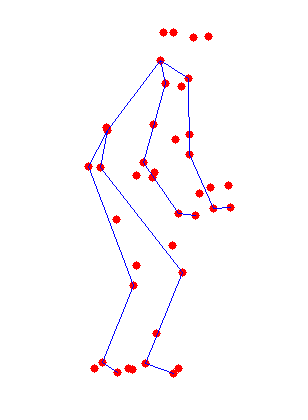}
\includegraphics[width=\w]{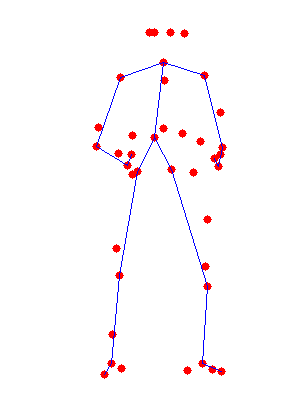}
\includegraphics[width=\w]{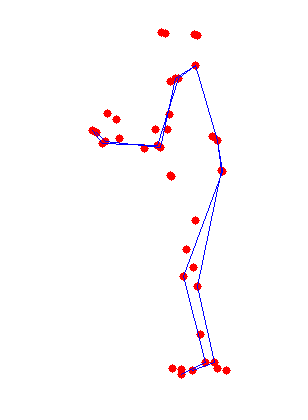}&
\includegraphics[width=\w]{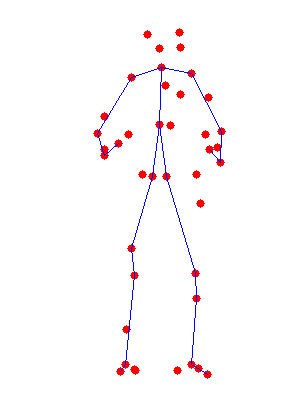}
\includegraphics[width=\w]{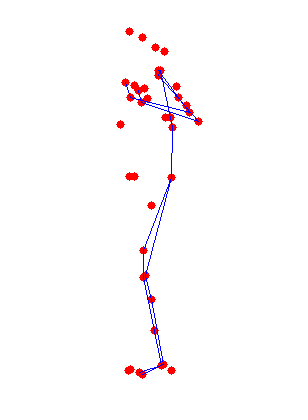}
\includegraphics[width=\w]{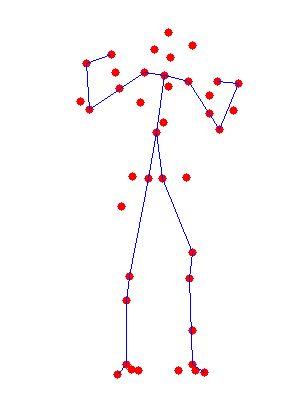}
\includegraphics[width=\w]{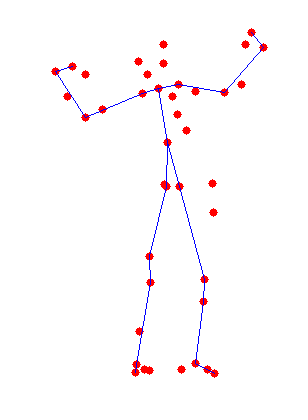}&
\includegraphics[width=\w]{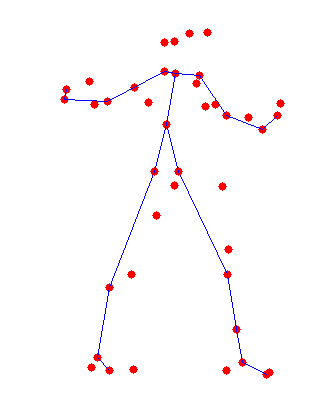}
\includegraphics[width=\w]{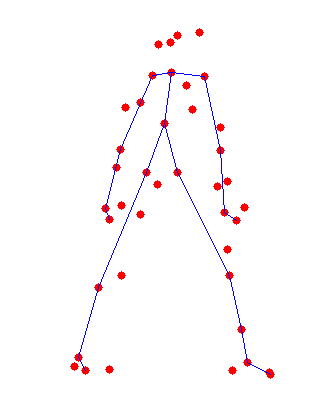}
\includegraphics[width=\w]{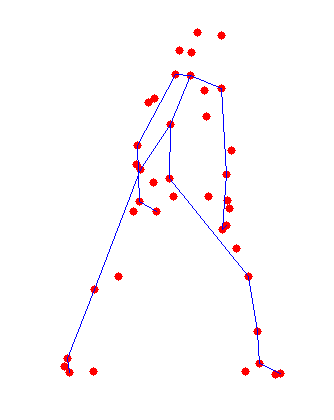}
\includegraphics[width=\w]{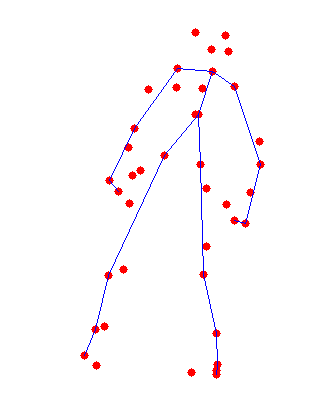}\\
\emph{Drink} & \emph{Pick-up} &
\emph{Stretch} & \emph{Yoga} \\
\hline
\end{tabular}
\end{center}
\caption{Four images from each of the MOCAP data sets.}
\label{fig:mocapshapes}
\def\w{42mm}
\begin{center}
\includegraphics[width=\w]{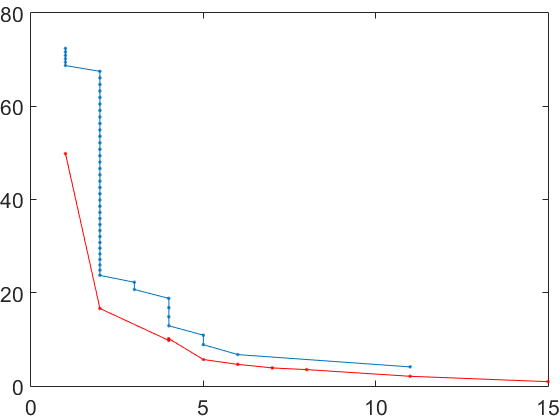}
\includegraphics[width=\w]{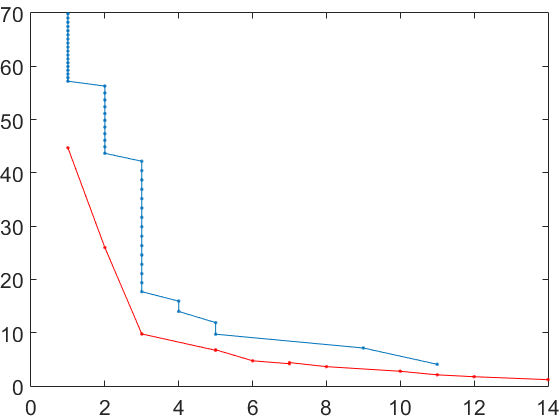}
\includegraphics[width=\w]{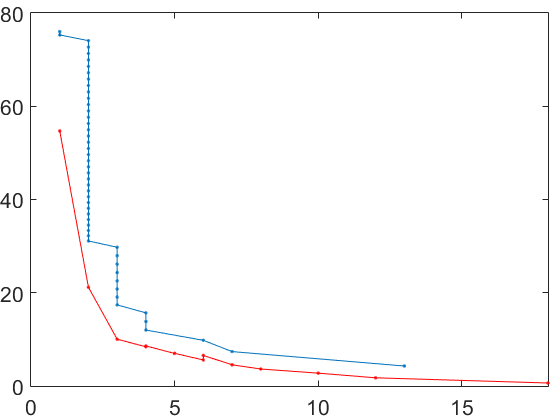}
\includegraphics[width=\w]{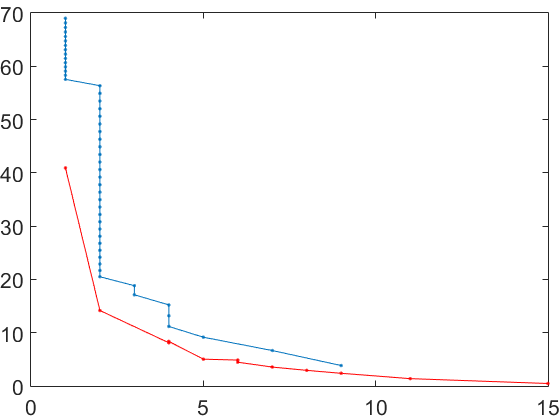}
\end{center}
\caption{Results obtained with \eqref{eq:deformobjrelax} and \eqref{eq:deformobjnuclear} for the four sequences.
Data fit $\|R X - M\|_F$ (y-axis) versus $\rank(X^\#)$ (x-axis) is plotted for various regularization strengths.
Blue curve uses $2 \sqrt{\mu} \|X^\#\|_*$ and red curve $r_\mu(\bfsigma(X^\#))$
with $\mu=1,...,50$.}
\label{fig:deformdatafit}
\begin{center}
\includegraphics[width=\w]{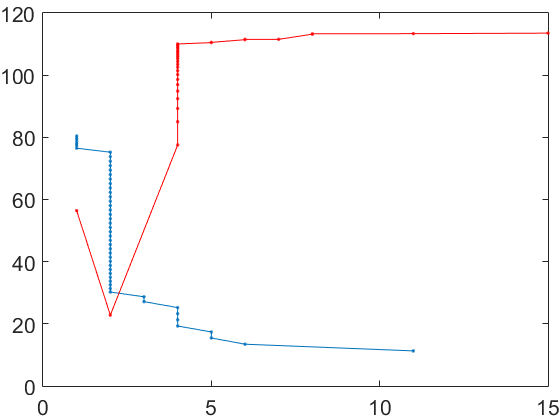}
\includegraphics[width=\w]{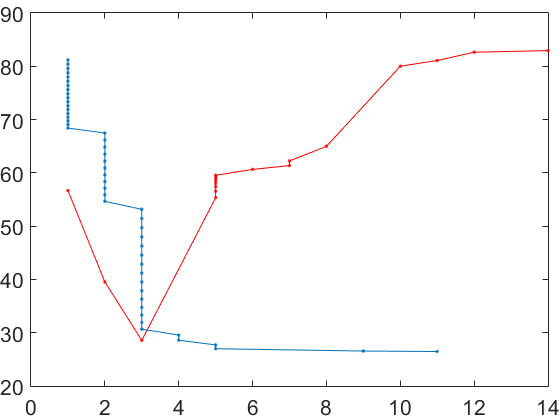}
\includegraphics[width=\w]{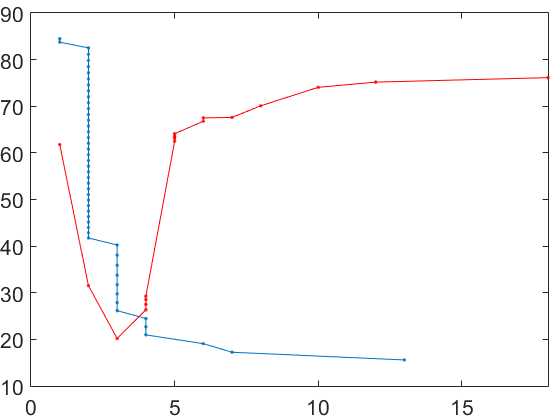}
\includegraphics[width=\w]{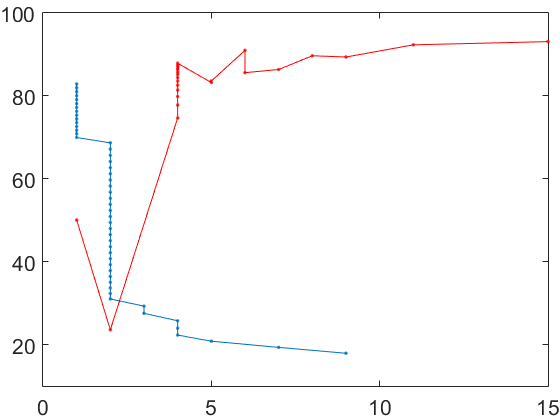}
\end{center}
\caption{Results obtained with \eqref{eq:deformobjrelax} and \eqref{eq:deformobjnuclear} for the four sequences.
Distance to ground truth $\|X - X_{gt}\|_F$ (y-axis) versus $\rank(X^\#)$ (x-axis) is plotted for various regularization strengths.
Blue curve uses $2 \sqrt{\mu} \|X^\#\|_*$ and red curve $r_\mu(\bfsigma(X^\#))$
with $\mu=1,...,50$.}
\label{fig:deformgtdist}
\begin{center}
\includegraphics[width=\w]{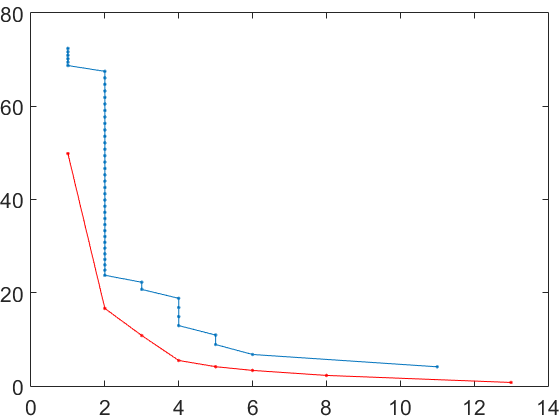}
\includegraphics[width=\w]{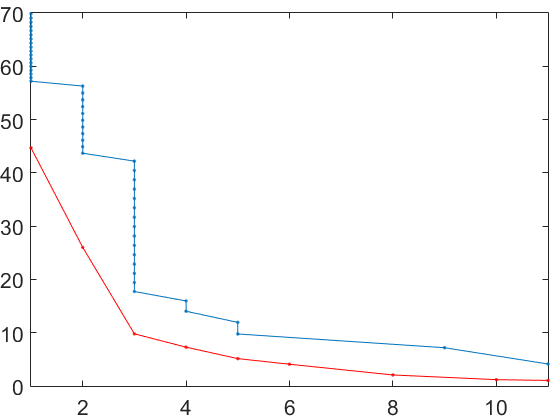}
\includegraphics[width=\w]{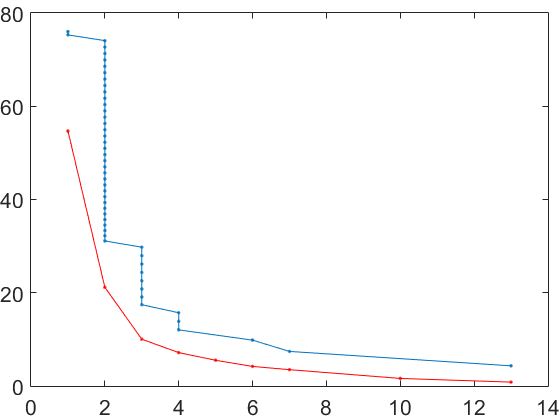}
\includegraphics[width=\w]{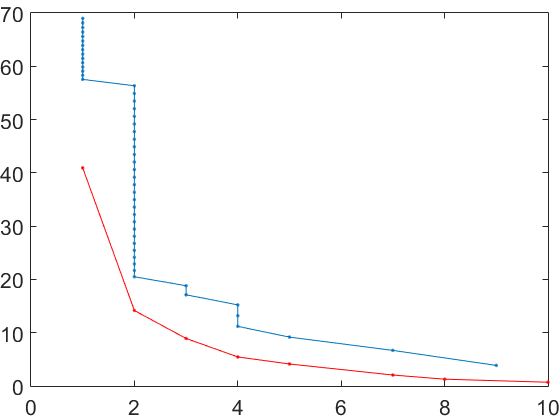}
\end{center}
\caption{Results obtained with \eqref{eq:deformobjrelaxD} and \eqref{eq:deformobjnuclearD} for the four sequences.
Data fit $\|R X - M\|_F$ versus $\rank(X^\#)$ is plotted for various regularization strengths.
Blue curve uses $2 \sqrt{\mu} \|X^\#\|_*$ and red curve $r_\mu(\bfsigma(X^\#))$
with $\mu=1,...,50$.}
\label{fig:deformdatafitD}
\begin{center}
\includegraphics[width=\w]{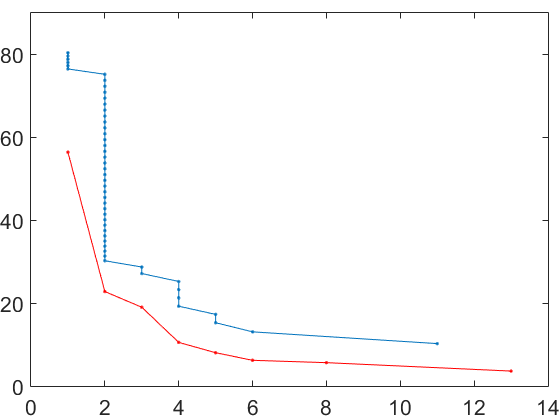}
\includegraphics[width=\w]{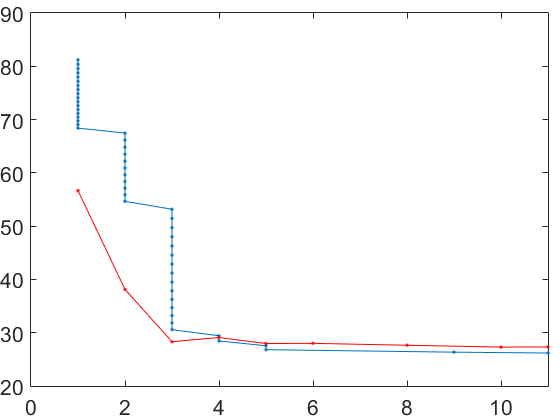}
\includegraphics[width=\w]{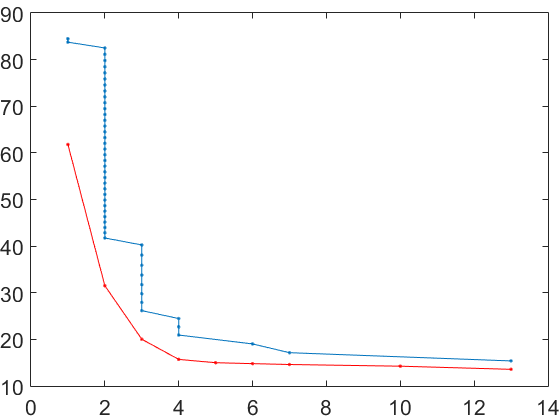}
\includegraphics[width=\w]{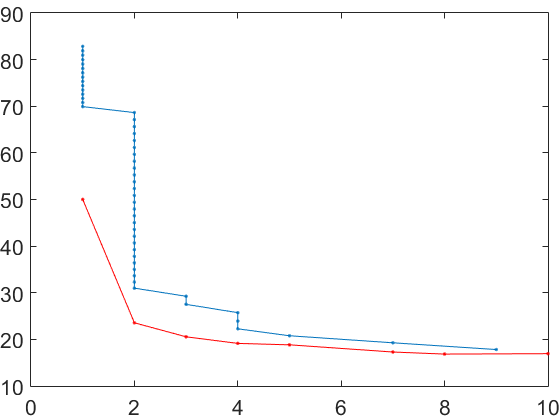}
\end{center}
\caption{Results obtained with \eqref{eq:deformobjrelaxD} and \eqref{eq:deformobjnuclearD}for the four sequences.
Distance to ground truth $\|X - X_{gt}\|_F$ (y-axis) versus $\rank(X^\#)$ (x-axis) is plotted for various regularization strengths.
Blue curve uses $2 \sqrt{\mu} \|X^\#\|_*$ and red curve $r_\mu(\bfsigma(X^\#))$
with $\mu=1,...,50$.}
\label{fig:deformgtdistD}
\end{figure*}

Given projections of a number 3D points on an object, tracked through several images, 
the goal of non-rigid SfM is to reconstruct the 3D positions of the points. Note that the object can be deforming throughout the image sequence. In order to make the problem well posed some knowledge of the deformation has to be included.
Typically one assumes that all possible object shapes are spanned by a low dimensional linear basis
\cite{bregler-etal-cvpr-2000}. Specifically, let $X_f$ be a $3 \times n$-matrix
containing the coordinates of the 3D points when image $f$ was taken.
Here column $i$ of $X_f$ contains the x-,y- and z-coordinates of point $i$.
Under the linearity assumption there is a set of basis shapes $B_k$, $k=1,...,K$ such that
\begin{equation}
X_f = \sum_{k=1}^{K} c_{fk} B_k.
\label{eq:shapebasis}
\end{equation}
Here the basis shapes $B_k$ are of size $3\times n$ and the coefficients $c_{fk}$ are scalars.
The projection of the 3D shape $X_f$ into the image is modeled by $x_f = R_f X_f$. The $2 \times 3$ matrix $R_f$ contains two rows from an orthogonal matrix which encodes camera orientation.
%Bregler \etal \cite{bregler-etal-cvpr-2000} showed that if we form the $2F \times n$ matrix $X$ by concatenating the matrices $x_f$, $f=1,...,F$ the resulting matrix can be written $R(C \otimes I_3)B$.

Dai \etal \cite{dai-etal-ijcv-2014} observed that \eqref{eq:shapebasis}
can be interpreted as a low rank constraint by reshaping the matrices.
First, let $X_f^\#$ be the $1 \times 3n$ matrix obtained by concatenation of the 3 rows $X_f$.
Second, let $X^\#$ be $F \times 3n$ with rows $X_f^\#$, $f=1,...,F$.
Then \eqref{eq:shapebasis} can be written
%\begin{equation}
$X^\# = CB^\#,$
%\end{equation}
where $C$ is the $F \times K$ matrix containing the coefficients $c_{fk}$ and $B^\#$ is a $K \times 3n$ matrix constructed from the basis in the same way as $X^\#$.
The matrix $X^\#$ is thus of at most rank $K$. Furthermore, the complexity of the deformation can be constrained by penalizing the rank of $X^\#$.

To define an objective function we let the $2F\times n$ matrix $M$ be the concatenation of all the projections $x_f$, $f=1,...,F$. Similarly we let the $3F \times n$ matrix $X$ be the concatenation of the $X_f$ matrices.
The objective function proposed by \cite{dai-etal-ijcv-2014} is then
\begin{equation}
\mu \rank (X^\#) + \|R X - M\|^2_F,
\label{eq:deformobj}
\end{equation}
where $R$ is a $2F \times 3F$ block-diagonal matrix containing the $R_f$, $f=1,...,F$ matrices.
Dai \etal proposed to solve \eqref{eq:deformobj} by replacing the rank penalty with $\|X^\#\|_*$.
In this section we compare this to our approach that instead uses $r_\mu (\bfsigma(X^\#))$.
We test the approach on the 4 MOCAP sequences \emph{Drink}, \emph{Pick-up}, \emph{Stretch} and \emph{Yoga} used in \cite{dai-etal-ijcv-2014}, see Figure~\ref{fig:mocapshapes}.
Note that the MOCAP data is generated from motions recorded using real motion-capture-systems and the ground truth is therefore not of low rank.
In Figure~\ref{fig:deformdatafit} we compare the two relaxations
\begin{equation}
r_{\mu} (\bfsigma(X^\#)) + \|R X - M\|^2_F
\label{eq:deformobjrelax}
\end{equation}
and
\begin{equation}
2\sqrt{\mu} \|X^\#\|_* + \|R X - M\|^2_F,
\label{eq:deformobjnuclear}
\end{equation}
for varying values of $\mu$. We plot the obtained data fit versus the obtained rank for $\mu=1,...,50$.
The stair case shape of the blue curve is due to the nuclear norm's bias to small solutions.
When $\mu$ is modified the strength of this bias changes and modifies the value of the data fit even if the modification is not big enough to change the rank.
In contrast the data fit seems to take a (roughly) unique value for each rank when using \eqref{eq:deformobjrelax}.

The relaxation \eqref{eq:deformobjrelax} consistently generates better data fit for all ranks and as an approximation of \eqref{eq:deformobj} it clearly performs better than \eqref{eq:deformobjnuclear}.
This is however not the whole truth. In Figure~\ref{fig:deformgtdist} we also plotted the distance to the ground truth solution. When the obtained solutions are not of very low rank \eqref{eq:deformobjnuclear} is generally better than \eqref{eq:deformobjrelax} despite consistently generating a worse data fit.
A feasible explanation is that when the rank is larger than roughly 3-4 there are multiple solutions with the same rank giving the same projections (witch also implies that the RIP \eqref{eq:matrisRIP} does not hold for this rank).
Note that in Figure~\ref{fig:deformdatafit} the data fit seems to take a unique value for every rank. In short; when the space of feasible deformations becomes too large we cannot uniquely reconstruct the object from image data without additional priors.
In contrast the ground truth distance can take several values for a given rank in Figure~\ref{fig:deformgtdist}. The nuclear norm's bias to solutions with small singular values seems to have a regularizing effect on the problem.

Dai \etal \cite{dai-etal-ijcv-2014} also suggested to further regularize the problem by penalizing derivatives of the 3D trajectories. For this they use a term $\|DX^\#\|_F^2$, where the matrix $D:\mathbb{R}^F \rightarrow \mathbb{R}^{F-1}$ is a first order difference operator.
For completeness we add this term and compare
\begin{equation}
 r_{\mu} (\bfsigma(X^\#)) + \|R X - M\|^2_F +\|DX^\#\|_F^2
\label{eq:deformobjrelaxD}
\end{equation}
and
\begin{equation}
2\sqrt{\mu} \|X^\#\|_* + \|R X - M\|^2_F+ \|DX^\#\|_F^2.
\label{eq:deformobjnuclearD}
\end{equation}
Figures~\ref{fig:deformdatafitD} and~\ref{fig:deformgtdistD} show the results. Our relaxation \eqref{eq:deformobjrelaxD} generally finds better data fit at lower rank than what \eqref{eq:deformobjnuclearD} does. Additionally, for low ranks \eqref{eq:deformobjrelaxD} provides solutions that are closer to ground truth.
When the rank increases most of the regularization becomes more dependent on the derivative prior leading to both methods providing similar results.
\begin{figure*}
\def\w{45mm}
\begin{center}
\begin{tabular}{ccc}
(a) & (b) & (c) \\
\includegraphics[width=\w]{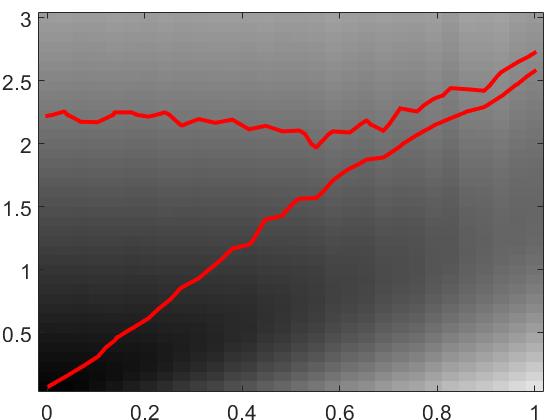} &
\includegraphics[width=\w]{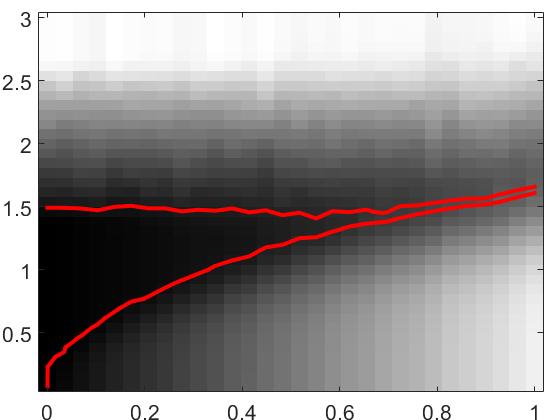} &
\includegraphics[width=\w]{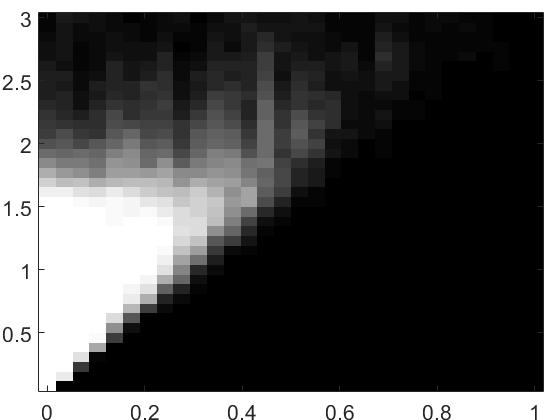} \\
%\end{tabular}
%\begin{tabular}{cc}
(d) & (e) \\
\includegraphics[width=\w]{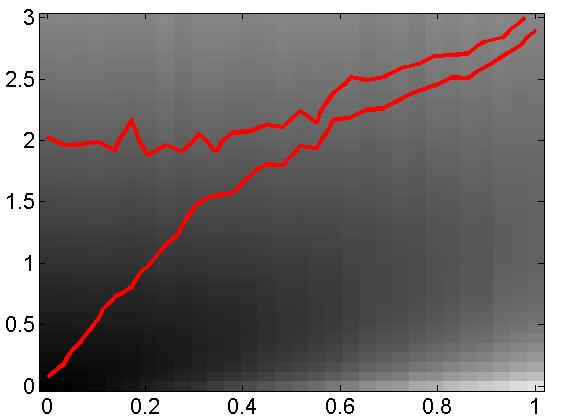} &
\includegraphics[width=\w]{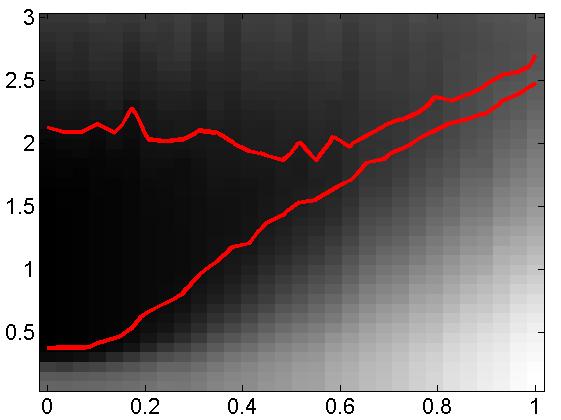} 
\end{tabular}
\end{center}
\caption{Sparse recovery results for varying noise level (x-axis) and regularization strength (y-axis). 
Top row: Random $200 \times 200$ $A$ with $\delta=0.2$.
Bottom row: Random $150 \times 200$ $A$ (and unknown $\delta$).
Plots (a) and (d) show the average distance between the $\ell_1$ regularized and the ground truth solutions for values of $\mu$ between $0$ and $3$.
(red curves marks the area where the obtained solution has $\supp(\x)=10$.)
Plots (b) and (e) show the average distance between \eqref{eq:ourvector} and the ground truth solutions.
Plot (c) shows the number of instances where our method could be verified to provide the global optima for $\delta=0.2$ (white = all, black = none).}
\label{fig:sparse_results}
\end{figure*}

\subsection{Sparse Recovery}\label{sec:sparsityexp}
We conclude the paper with a synthetic experiment on sparse recovery.
For Figure~\ref{fig:sparse_results} (a)-(c) we randomly generated problem instances for sparse recovery. Each instance uses a matrix $A$ of size $200 \times 200$ with $\delta = 0.2$ which was generated by first randomly sampling the elements of a matrix $\tilde{A}$ a Gaussian $\mathcal{N}(0,1)$ distribution. The matrix $A$ was then constructed from $\tilde{A}$ by modifying the singular values to be evenly distributed between $\sqrt{1-\delta}$ and $\sqrt{1+\delta}$.
To generate a ground truth solution and a $\bb$ vector we then randomly select values
for $10$ nonzero elements of $\x$ and computed $\bb=A\x+\epsilon$, where all elements of $\epsilon$ are $\mathcal{N}(0,\sigma^2)$.

The averaged results (over 50 random instances for each $(\sigma,\mu)$ setting) are shown in 
Figure~\ref{fig:sparse_results} (a)-(c). 
Similar to the matrix case it is quite clear that the $\ell_1$ norm (a) suffers from shrinking bias. It consistently gives the best agreement with the ground truth data for values of $\mu$ that are not big enough to generate low cardinality. 
In contrast, \eqref{eq:ourvector} gives the best fit at the correct cardinality for all noise levels. This fit was consistently better than that of \eqref{eq:l1} for all noise levels.
In Figure~\ref{fig:sparse_results} (c) we show the fraction of problem instances that could be verified to be optimal. 

In Figure~\ref{fig:sparse_results} (d) and (e) we tested the case where the elements of an $m \times n$ matrix $A$ are sampled from $\mathcal{N}(0,\frac{1}{m})$ \cite{candes-etal-cpam-2006}.
Here we let $A$ be random $150 \times 200$ matrices and generated the ground truth solution and $\bb$ vector as described previously. Here \eqref{eq:ourvector} consistently outperformed \eqref{eq:l1} which exhibits the same tendency to achieve a better fit for non-sparse solutions.

\section{Conclusions}
In this paper we studied the local minima of a non-convex rank/sparsity regularization approach.
Or main results show that if a RIP property holds then the stationary points are often well separated. This gives an explanation as to why many non-convex approaches such as \cite{daubechies-etal-cpam-2010,candes-etal-2008,oymak-etal-2015,mohan2010iterative} can be observed to work well in practice.
Our experimental evaluation verifies that the proposed approach often recovers better solutions than standard convex counterparts, even when the RIP constraint fails.

{\small
\bibliographystyle{plain}
%\bibliography{newlib}

}

\end{document}